\documentclass{article}
\usepackage[utf8]{inputenc}
\usepackage{tikz}
\usetikzlibrary{calc, intersections}

\usepackage{amsmath,amsthm,amsfonts}
\usepackage{verbatim}
\usepackage{mathtools}
\mathtoolsset{showonlyrefs}
\usepackage[disable]
{todonotes}
\newcommand{\marginnote}{\todo[color=blue!20, linecolor=blue, bordercolor=blue]}

\theoremstyle{plain}
\newtheorem{theorem}{Theorem}

\newtheorem{corollary}[theorem]{Corollary}

\newtheorem{observation}[theorem]{Observation}

\theoremstyle{definition}
\newtheorem{definition}[theorem]{Definition}

\theoremstyle{remark}
\newtheorem{remark}[theorem]{Remark}

\newcommand{\defeq}{\vcentcolon=}

\newcommand{\set}[2][]{#1\{ #2 #1\}}

\newcommand{\abs}[2][]{#1\vert #2 #1\vert}

\newcommand{\size}[2][]{#1\vert #2 #1\vert}

\newcommand{\R}{\mathbb{R}}
\newcommand{\E}{\mathop{\mathrm{E}}}
\newcommand{\Var}{\mathop{\mathrm{Var}}}

\newcommand{\B}{\mathcal{B}}
\newcommand{\C}{\mathcal{C}}
\newcommand{\F}{\mathbb{F}}
\renewcommand{\L}{\mathcal{L}}

\renewcommand{\u}{ to ++(1,1) node[draw=none] {}}
\renewcommand{\d}{ to ++(1,-1) node[draw=none] {}}
\newcommand{\dl}{ to ++(-1,-1) node[draw=none] {}}
\newcommand{\ul}{ to ++(-1,1) node[draw=none] {}}

\title{Lattice Path Matroids: Negative Correlation and Fast Mixing}
\author{Emma Cohen\footnote{School of Mathematics, Georgia Tech, Atlanta, GA 30332-0160} \and Prasad Tetali\footnote{School of Mathematics and School of Computer Science, Georgia Tech, Atlanta, GA 30332-0160}\and Damir Yeliussizov\footnote{Department of Computer Science and Electrical Engineering, Kazakh-British Technical University, Almaty, Kazakhstan}}
\begin{document}

%%%%%%%%%%%%%%%%%%%%%%%%%%
%%%%%%%%%%%%%%%%%%%%%%%%%%
        \maketitle
%%%%%%%%%%%%%%%%%%%%%%%%%%
%%%%%%%%%%%%%%%%%%%%%%%%%%

\begin{abstract}
Catalan numbers arise in many enumerative contexts as the counting sequence of combinatorial structures. In this work, we consider natural Markov chains on some of the realizations of the Catalan sequence. %, and derive estimates on the mixing time of the corresponding Markov chains. 
While our main result is in deriving an $O(n^2 \log n)$ bound on the mixing time in $L_2$ (and hence total variation) distance for the random transposition chain on Dyck paths, we raise several open questions, including the optimality of the above bound. The novelty in our proof is in establishing a certain negative correlation property among random bases of lattice path matroids, including the so-called {\em Catalan matroid} which can be defined using Dyck paths.
\end{abstract}

%%%%%%%%%%%%%%%%%%%%%%%%%%%%%%%%%%%%%%%%
\section{Introduction} \label{sec:intro}
%%%%%%%%%%%%%%%%%%%%%%%%%%%%%%%%%%%%%%%%

There are several combinatorial structures whose enumeration is given by the Catalan sequence: $1, 2, 5, 14, 42, \ldots$, where the $n$th term (for $n\ge 1$) is given by the Catalan number $C_n = \frac{1}{n+1}\binom{2n}{n}$. Some examples include the set of all triangulations of a regular polygon of $n+2$ sides, the set of non-crossing partitions of an $n$-set (the lattice on which is of much interest to researchers in free probability), and the set of balanced parentheses, with $n$ left and $n$ right parentheses---or equivalently, described as {\em Catalan strings} $x\in \{\pm 1\}^{2n}$ of $n$ $1$'s and $n$ $-1$'s with non-negative partial sums, $\sum_{i=1}^j x_i \ge 0$, for all $j$. This last structure (of Catalan strings) is also known as the set of Dyck paths, and is visualized as lattice paths of $n$ up-steps and $n$ down-steps, representing the $1$'s and the $-1$'s respectively.

For any given $n\ge 1$, generating a Catalan structure (such as triangulation or a Dyck path) uniformly at random from the set of $C_n$ many, is a straightforward task, and can be done in time linear in $n$, not too unlike generating a permutation uniformly at random  (out of the $n!$ many permutations) of $n$ distinct letters. 
However, the study (of convergence to equilibrium) of particular Markov chain Monte Carlo algorithms which yield a random Catalan structure is seemingly much more interesting. The inspiration to this endeavor in part stems from  an open problem of David Aldous~\cite{Aldous94}, who conjectured that the random walk on triangulations of a polygon on $n$ sides, performed using uniform diagonal flips, ought to take time roughly  $n^{3/2}$ (up to factors logarithmic in $n$).
Despite much effort by various researchers in this topic over several years, the best known bounds for the mixing time of the chain on triangulations remain those of  McShine-Tetali \cite{McShine1999}, who proved an upper bound of $O(n^4)$ on the relaxation time (also known as the inverse spectral gap), and Molloy-Reed-Steiger \cite{MRS98} who showed  $\Omega(n^{3/2})$ as a lower bound.

In the work, we explore a (different) Markov chain on a different realization of the Catalan number: we consider the set of Dyck paths, described using binary strings $x \in \{\pm 1\}^{2n}$, as mentioned above, and consider the natural local move of exchanging $x_i$ and $x_j$, for a pair $(i,j)$ chosen uniformly at random from among all possible pairs, {\em provided} that the resulting string is a (valid) Catalan string; else we reject such a move. Note that, without the rejection step, we would have a random transposition chain on the set of $\binom{2n}{n}$ many binary strings with an equal number of $1$'s and $-1$'s. It is well-known (see e.g., \cite{Diaconis1987})---sometimes under the name \emph{Bernoulli-Laplace model}---that $O(n\log n)$ such random transpositions are necessary and sufficient to reach (close to) equilibrium. It is also well-known that the so-called random transposition shuffle of $n$ distinct cards mixes in $O(n\log n)$ time. However, the ``sea-level'' constraint of having the partial sums of the (Catalan) string stay non-negative seems to thwart any type of straight forward (or otherwise) analyses, despite the best efforts of several experts in the field. 

Our main result is an $O(n^2\log n)$ upper bound on the mixing time of the random transposition walk on Dyck paths of length $n$.  The proof relies on a rephrasing of the walk as a basis exchange walk on a balanced matroid.  In Section \ref{sec:matroids} we describe the Catalan matroid (due to Ardila \cite{Ardila2003})  and the prior work \cite{Feder1992, Jerrum2004, Jerrum-Son} regarding mixing times for the basis exchange walk on balanced matroids.  All that remains to bound the mixing time of our Dyck path walk is the proof in Section \ref{sec:main-result} that the Catalan matroid is balanced.  In Section \ref{sec:lower-bounds} we give a short proof of a mixing time lower bound for Dyck random transpositions and, finally, in Section \ref{sec:open-questions} we explore the possible consequences of this result for other walks on Catalan structures and pose some open questions.

\subsection{Random walks on lattice paths}
%\TODO{Make sure the notation here matches what is used in the main proof.}
Throughout, for integers $a\leq b$ we let $[a,b] \defeq \set{a, a+1,\dots\,b}$ denote the discrete closed interval and $[a] \defeq [1,a]$.

A \emph{lattice path} of length $m$ is a string $P\in\set{\pm1}^m$ of \emph{up-steps} ($+$) and \emph{down-steps} ($-$).  The \emph{height} of $P$ at index $i$ is $h_i(P) = \sum_{j=1}^i P_i$, and we can draw $P$ on the grid as the graph of $f_P(i) = h_i(P)$.  That is, we draw a path starting at $(0,0)$ and taking up-steps $(1,1)$ and down-steps $(1,-1)$.
% \begin{figure}
%     \centering
%     \begin{tikzpicture}[scale=.3]
%     \draw[draw=black!5, very thin] (0,-3) grid (14,4);
%     \draw[draw=black!30] (0,0) to (14,0);
%     \draw[draw=black!30] (0,-3) to (0,4);
%     \draw[very thick] (0,0)  \u\d\u\u\u\d\d\d\d\d\u\u\d\d;
%     \foreach[count=\i] \x in {+,-,+,+,+,-,-,-,-,-,+,+,-,-} {
%         \node at (\i-.5, -3.5) {$\x$};
%     }
%     \end{tikzpicture}\qquad
%     \begin{tikzpicture}[scale=.3]
%     \draw[draw=black!5, very thin] (0,-3) grid (14,4);
%     \draw[draw=black!30] (0,0) to (14,0);
%     \draw[draw=black!30] (0,-3) to (0,4);
%     \draw[very thick] (0,0)  \u\d\u\u\u\d\u\d\d\d\u\u\d\d;
%     \foreach[count=\i] \x in {+,-,+,+,+,-,+,-,-,-,+,+,-,-} {
%         \node at (\i-.5, -3.5) {$\x$};
%     }
%     \end{tikzpicture}
%     \caption{A non-Dyck lattice path (left), and a Dyck path (right).}
% \end{figure}
We will refer to a lattice path $P$ of length $m$ as a lattice path \emph{from $(0,0)$ to $(m, h_m(P))$}.

From this we can define a partial order on the set of lattice paths from $(0,0)$ to $(m,2r-m)$ by letting $P \leq Q$ whenever $h_i(P) \leq h_i(Q)$ for all $i\in [m]$.  Note that $P\leq Q$ if and only if $q_i \leq p_i$ for all $i\in[r]$, where $p_i$ (resp. $q_i$) is the index of the $i$th up-step in $P$ (resp. $Q$).

\begin{definition}
    A \emph{Dyck path} of length $2n$ is a lattice path $P$ from $(0,0)$ to $(2n,0)$ with $h_i(P) \geq 0$ for all $i\in[2n]$.
\end{definition}

That is, the Dyck paths are precisely the paths $P$ from $(0,0)$ to $(0,2n)$ with $P \geq (+-)^n$. It is a standard result that the number of Dyck paths of length $2n$ is the Catalan number $C_n = \frac{1}{n+1}\binom{2n}{n}$.

A natural class of random walks on lattice paths from $(0,0)$ to $(m,h)$ is the \emph{transposition walk}, which at each step picks random indices $i,j\in[m]$ and swaps the steps of $P$ at those indices.
% \begin{figure}
%     \caption{The transposition corresponding to indices $(i,j) = (2,7)$ in a lattice path.}
% \end{figure}
Indeed, if $i$ and $j$ are chosen uniformly independently at random this is just the Bernoulli-Laplace model on $\binom{[m]}{r}$, where $r = (h+m)/2$ is the number of up-steps in every such path, and has been studied extensively.  If we pick $i,j$ to be uniform among \emph{adjacent} pairs of indices, we get a different walk, whose mixing properties were determined in \cite{Wilson2004}.

But what if we consider lattice paths from $(0,0)$ to $(2n, 0)$ and restrict the walk to Dyck paths?  That is, from any Dyck path we pick $i,j$ according to one of these rules but only perform the move if the resulting path is again a Dyck path. (We will call these walks the \emph{Dyck random transpositions walk} and \emph{Dyck adjacent transpositions walk}, respectively.)  

Indeed, the analysis of \cite{Wilson2004} extends also to this case, yielding the same upper bound for the mixing time of the Dyck adjacent transpositions walk.  On the other hand, the Dyck random transpositions walk has evaded such precise analysis.

%%%%%%%%%%%%%%%%%%%%%%%%%%%%%%%%%%%%%%%%
\section{Lattice path matroids} \label{sec:matroids}
%%%%%%%%%%%%%%%%%%%%%%%%%%%%%%%%%%%%%%%%

%%%%%%%%%%
\subsection{Matroids and the basis exchange walk}
%%%%%%%%%%

Recall that a nonempty set $\B\subseteq 2^U$ is the set of \emph{bases} of a \emph{matroid} $M = (U, \B)$ if the following \emph{basis exchange axiom} holds:
\begin{quote}
    \textbf{Matroid Basis Exchange Axiom.} For any bases $A,B\in \B$ and every $e\in A\setminus B$ there exists $f\in B\setminus A$ such that $A\setminus\set{e}\cup\set{f}\in \B$ is a basis.
\end{quote}
Among other things, this axiom guarantees that all bases have the same cardinality, which is the \emph{rank} of $M$.

We will also make use of two dual operations on matroids: contraction and deletion.
\begin{definition}
    For a matroid $M = (U, \B)$ and an element $e\in M$, the matroid \emph{$M$ contract $e$} is $M_e = (U, \B_e)$, where $\B_e = \set{B\in \B: e\in B}$.  Similarly, \emph{$M$ delete $e$} is $M^e = (U, \B^e)$, where $B^e = \set{B\in \B: e\not\in B}$.\footnote{Note that this definition differs slightly from the usual one, in which the element $e$ being contracted or deleted is removed from the ground set.  In our case it will be convenient to leave $e$ in place, so as to more easily identify $\B_e\subseteq \B$ and preserve identities such as $\B = \B_e \cup \B^e$ (with $\B_e\cap \B^e = \emptyset$).}
  
    A \emph{minor} of $M$ is any matroid which can be obtained from $M$ through a series of contractions and deletions.
\end{definition}

The order in which contractions and deletions are performed does not matter, so we will write $M_I^J$ for the matroid obtained from $M$ by contracting the elements in $I$ and deleting the elements in $J$, and $\B_I^J$ will denote the set of bases of $M_I^J$.

Given a matroid $M$, Feder and Mihail \cite{Feder1992} study the following \emph{basis exchange walk} on the state space $\B$ of bases: 
\begin{quote} 
    From state $B\in\B$, pick uniformly and independently at random elements $a\in U$ and $b\in B$ and move in the next step to $B' = B\cup\set{a}\setminus\set{b}$ if $B'\in \B$ (else remain at $B$). 
\end{quote}
The matroid basis exchange axiom guarantees that this walk is ergodic, and since it is symmetric its stationary distribution is uniform.  It is tentatively conjectured that the basis exchange walk is fast for any matroid (i.e., the mixing time is bounded by some polynomial in $m = \size{U}$), but there is little evidence in favor of this.  

On the other hand, \cite{Feder1992} introduced the notion of \emph{balanced} matroids, which capture the notion that for a randomly chosen basis (of the matroid or any of its minors), conditioning on the occurrence of one element (in the basis) makes the occurrence of any other less probable.  They show that for the case of balanced matroids the walk is indeed rapidly mixing, by using decomposition techniques to bound its spectral gap.

\begin{definition}
    A matroid $M = (U,\B)$ is \emph{negatively correlated} if for every pair of distinct elements $e,f\in U$  
    \begin{align} \label{eqn:negative-correlation}
        \frac{\size{\B_e}}{\size{\B}} \geq \frac{\size{\B_{ef}}}{\size{\B_f}}. 
    \end{align}
    The matroid $M$ is \emph{balanced} if $M$ and all of its minors are negatively correlated.
\end{definition}
Negative correlation is equivalent to the very natural condition that for a uniform random basis $B\in\B$, $\Pr[e\in B] \geq \Pr[e\in B | f\in B]$. 

Indeed, many common classes of matroids are balanced, including uniform matroids (whose bases are all size-$r$ subsets of the ground set), graphic matroids (with the ground set being the edges of a connected graph and the bases being spanning trees of the graph), matroids of rank $\leq 3$, and regular matroids (matroids which can be represented over every field) \cite{Feder1992,Choe2006}.  We will rely on the following extension of Feder and Mihail's result for balanced matroids, due to Jerrum and Son.

\begin{theorem}[\cite{Jerrum-Son}]\label{thm:balanced-mixing}
    The spectral gap $\lambda$ and log-Sobolev constant $\alpha$ for the basis exchange walk on a balanced matroid $M$ of rank $r$ on a ground set of size $m$ are lower bounded by
    \begin{align}
        \lambda \geq \frac{2}{mr} \qquad\text{and}\qquad \alpha \geq \frac{1}{2mr}.
    \end{align}
\end{theorem}

In particular, via standard mixing time bounds (see \cite{Diaconis1996}), the bound on the log-Sobolev constant implies that the mixing time of the basis exchange walk on a balanced matroid is at most $O(mr\log\log\size{\B})$.  

\begin{remark}
Due to these standard bounds and the result for balanced matroids the proof of our main result never explicitly refers to the definitions of $\lambda$ and $\alpha$.  These definitions can, however, be found in Section \ref{sec:lower-bounds}, where they are used to give a lower bound on the mixing time for the Dyck random transposition chain.
\end{remark}

%%%%%%%%%%
\subsection{The Catalan matroid and other lattice path matroids}
%%%%%%%%%%
To aid our setting of the Catalan transposition walk in terms of matroids, we refer to an observation of Ardila \cite{Ardila2003} that the set of Dyck paths can be thought of as a matroid.

\begin{definition}
    The \emph{Catalan matroid} of order $n$ is $\C(n) = ([2n], \B(n))$, where the elements of $\B(n)$ are the index-sets of up-steps in Dyck paths of length $2n$.
\end{definition}

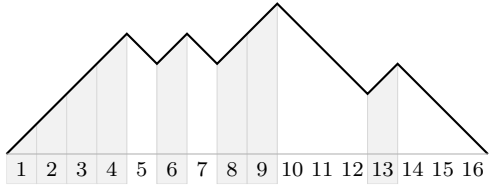
\begin{figure}
    \centering
    \begin{tikzpicture}[scale=.4]
        \begin{scope}
        \path[clip] (0,-1) to (0,0) \u \u \u \u \d \u \d \u \u \d \d \d \u \d \d \d to (16, -1);
        \foreach \x in {1,2,3,4,6,8,9,13}{
            \draw[fill=black!5, draw=black!15] (\x-1,-1) rectangle (\x,6);
        }
        \end{scope}
        \draw[thick] (0,0) \u \u \u \u \d \u \d \u \u \d \d \d \u \d \d \d;
        \foreach \x in {1,...,16}{
            \node at (\x-.5, -.5) {\footnotesize$\x$};
        }
        \draw[thin, draw=black!30] (0,0) to (16,0);
    \end{tikzpicture}
    \caption{The Dyck path $(++++-+-++---+---)$ above corresponds to the basis $\set{1,2,3,4,6,8,9,13}$ of the Catalan matroid of order $8$.}
\end{figure}

Ardila notes that this is precisely the transversal matroid for the set system $\mathcal{S} = \set{[1],[3],\dots,[2n-1]}$ (i.e., its bases are the systems of distinct representatives of $\mathcal{S}$).  In short, the representative of the set $[2i-1]$ will be the index of the $i$th up-step in the corresponding Dyck path (although this assignment of distinct representatives need not be unique).  Indeed,

\begin{observation} 
    The basis exchange walk on the Catalan matroid is exactly the random transposition walk on Dyck paths.  
\end{observation}

In other words, given Theorem \ref{thm:balanced-mixing}, to obtain a mixing time bound for the Dyck transposition walk it suffices to show that the Catalan matroids are balanced.

To this end, it would be convenient if the Catalan matroids belonged to some known class of balanced matroids.  As noted above, the main class of matroids known to be balanced is \emph{regular matroids}, but Ardila notes that $\C(n)$ is not representable over any $\F_q$ for $q\leq n-2$ and thus is not regular.  Transversal matroids also need not be balanced in general: Choe and Wagner \cite{Choe2006} give a transversal matroid of rank 4 which is not balanced.

Thus the bulk of our work here will be to show that the Catalan matroid is balanced, and from this the main mixing result will follow immediately.  For our later discussion it will help to work with a minor-closed class of matroids (which the Catalan matroids certainly are not).  Bonin and de Mier (\cite{Bonin2006}) discuss the following class of \emph{lattice path matroids}, which generalize the Catalan matroids by allowing \emph{any} pair of bounding paths.

\begin{definition}
    For two lattice paths $A\leq B$ from $(0,0)$ to $(m, 2r-m)$, consider the set $L$ of lattice paths $P$ from $(0,0)$ to $(m, 2r-m)$ with $A\leq P \leq B$.  The \emph{lattice path matroid} $\mathcal{L}[A,B]$ (which is of rank $r$ on ground set $[m]$) has as its bases the index sets of up-steps of paths in $L$.
\end{definition}

Although it is not immediately obvious that $\mathcal{L}[A,B]$ is a matroid, \cite{Bonin2006} observes that in fact the lattice path matroid $\L[A,B]$ is the transversal matroid of the set system $\set{[a_1,b_1], [a_2,b_2], \dots, [a_r, b_r]}$, where $a_i$ (resp. $b_i$) is the index of the $i$th up-step in $A$ (resp. $B$).  Indeed, it is shown in that paper that the transversal matroids of a set systems $\set{[a_1,b_1],[a_2,b_2],\dots,[a_r,b_r]}$ with $a_1 \leq a_2 \leq \dots \leq a_r$ and $b_1 \leq b_2 \leq \dots \leq b_r$ are precisely the lattice path matroids. (The set system of this form corresponding to a given lattice path matroid may not be unique; for example, the Catalan matroid is the transversal matroid for both $\set{[1],[3],\dots,[2n-1]}$ and $\set{[1],[2,3],\dots,[n,2n-1]}$.)  The class of basis exchange walks on lattice path matroids also includes the (unconstrained) Bernoulli-Laplace model as the basis exchange walk on a uniform matroid.

% \begin{figure}
%     \centering
%     \begin{tikzpicture}[scale=.3, thick]
%         \begin{scope}
%             \draw[clip] (0,0) \u \u \d \u \d \u \u \u \d \d \d \d \d 
%                               \dl\dl\ul\ul\ul\dl\ul\dl\dl\dl\ul\ul\ul;
%           %\draw[thick] (0,0)\d \d \d \u \u \u \d \u \d \d \d \u \u;
%             \draw[draw=black!15, rotate=-45, scale=sqrt(2), thin] (0,0) grid (7,6);
%         \end{scope}
%         \draw[very thick] (0,0)\u \d \d \d \u \u \u \u \d \u \d \d \d;
%         \node at (6,3) {$B$};
%         \node at (9,-2){$A$};
%     \end{tikzpicture}
%     \qquad
%     \begin{tabular}{c c c c c c}
%      1 & 5 & 6 & 7 & 8 & 10  \\
%      $[1,4]$&$[2,5]$&$[4,6]$&$[6,8]$&$[7,12]$&$[8,13]$
%     \end{tabular}
%     \caption{A lattice path matroid and its corresponding set system $\mathcal{S}$.  The matroid consists of all paths which stay between $A$ and $B$.  One such path is shown, along with the index set of its up-steps as a transversal of $\mathcal{S}$.}
% \end{figure}

In addition to showing this correspondence, Bonin and de Mier show that the class of lattice path matroids is closed under taking minors and duals.  They also define a smaller (in fact, minimal) minor-closed class of matroids containing $\C_n$, which they call \emph{generalized Catalan matroids}, consisting of the lattice path matroids $M(A,B)$ where $A = (+)^r (-)^{m-r}$ is maximal among all paths from $(0,0)$ to $(m, 2r-m)$.  They give several nice properties of generalized Catalan matroids, which we will not go into here as our analysis covers all lattice path matroids.  Indeed, Sohoni \cite{Sohoni1999} has actually already shown the balanced property of generalized Catalan matroids, although he calls them \emph{Schubert matroids} and does not link his result to rapid mixing of the Dyck random transposition chain (although he does mention the Dyck adjacent transposition chain).  Our proof applies to the more general class of lattice path matroids.

Given that the class of lattice path matroids is minor-closed, it suffices for our main result to show that lattice path matroids are negatively correlated, but it is worth noting that the same method can also be used to show negative correlation directly for any minor of a lattice path matroid.

%%%%%%%%%%%%%%%%%%%%%%%%%%%%%%%%%%%%%%%%
\section{Mixing bound for random transpositions} \label{sec:main-result}
%%%%%%%%%%%%%%%%%%%%%%%%%%%%%%%%%%%%%%%%

To show negative correlation for lattice path matroids, it will be convenient to use an equivalent formulation of \eqref{eqn:negative-correlation}, which is easily obtained by repeatedly applying the identity $\size{\B} = \size{B_e} + \size{B^e}$,
\begin{align}
    \size{\B_{ef}} \size{\B^{ef}} \leq \size{\B_{e}^{f}} \size{\B_{f}^{e}}\,. \label{neg-cor-equiv}
\end{align}
Now we are ready to prove

\begin{theorem}
    For every pair of lattice paths $A\leq B$ from $(0,0)$ to $(m, 2r-m)$ the lattice path matroid $\L[A,B]$ is negatively correlated.
\end{theorem}

\begin{proof}
Let $\L[A,B] = ([m], \B)$ be any lattice path matroid.

To prove inequality \eqref{neg-cor-equiv} for every pair $e < f\in [m]$, we will construct an injective map
\[\varphi_{ef} : \B_{ef} \times \B^{ef} \to \B_{e}^{f} \times \B_{f}^{e}.\] 
Note that we can associate members of these sets with lattice paths in $\L[A,B]$ by
\begin{itemize}
\item $\B_{ef}$: paths with up-steps at indices $e,f$,
\item $\B^{ef}$: paths with down-steps at indices $e,f$,
\item $\B_{e}^{f}$: paths with an up-step at index $e$ and a down-step at index $f$, and
\item $\B_{f}^{e}$: paths with a down-step at index $e$ and an up-step at index $f$.
\end{itemize}

Let $P \in \B_{ef}$ and $Q \in \B^{ef}$ be lattice paths and consider the following cases. (The figures shown are for the Catalan matroid.)
%\TODO{Fix the illustrations to account for any lattice path.}

\begin{description}
\item[Case 1.] Suppose the paths $P,Q$ intersect (without necessarily crossing) in the region (I) between $e$ and $f$.  Note that this includes all cases where the path $P$ is below $Q$ at $e$ and above $Q$ at $f$ or vice versa.  Take the first such intersection point $x$ in (I) and switch the paths $P, Q$ after $x$ to obtain new paths $P' \in \B_{e}^{f}, Q' \in \B_{f}^{e}$ as shown in Figure \ref{fig:case1}.  Set $\varphi_{ef}(P, Q) = (P', Q').$

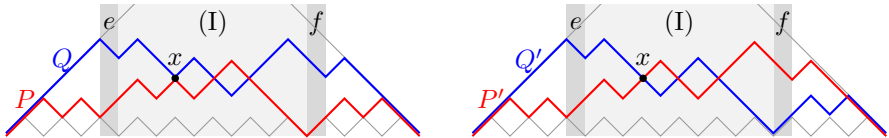
\begin{figure}[ht]
  \centering % CASE 1
  \begin{tikzpicture}[scale=.25]
    \path[use as bounding box] (0,0) rectangle (22,7);
    \fill[fill=black!15] (5,0) rectangle ++(1,7) +(-.5,-1) node {$e$};
    \fill[fill=black!15] (16,0) rectangle ++(1,7) +(-.5,-1) node {$f$};
    \fill[fill=black!5] (6,0) rectangle (16,7);
    \node at (11,6) {(I)};
    
    \begin{scope}
        \path[clip] (0,0) rectangle (22,7);
        \draw[thin, draw=black!40] (0,0.2) \u \u \u \u \u \u \u \u \u \u \u \d \d \d \d \d \d \d \d \d \d \d;
        \draw[thin, draw=black!40] (0,0)   \u \d \u \d \u \d \u \d \u \d \u \d \u \d \u \d \u \d \u \d \u \d;
    \end{scope}

  	%                       e         x                      f                
    \draw[blue, thick] (0,0.15)  % Q
    		\u \u \u \u \u \d \u \d \d coordinate (x) \u \d \d \u \u \u \d \d \u \d \d \d \d;
    \draw[red, thick] (0,0)  % P 
    		\u \u \d \u \d \u \u \d \u \d \u \u \d \d \d \d \u \u \d \u \d \d;
    \node[red] at (1,2) {$P$};
    \node[blue] at (3,4) {$Q$};
    \path (x) +(0,-0.075) node [label=above:{$x$}, fill, circle, inner sep=1pt] {};
  \end{tikzpicture}\qquad
  \begin{tikzpicture}[scale=.25]
    \path[use as bounding box] (0,0) rectangle (22,7);
    \fill[fill=black!15] (5,0) rectangle ++(1,7) +(-.5,-1) node {$e$};
    \fill[fill=black!15] (16,0) rectangle ++(1,7) +(-.5,-1) node {$f$};
    \fill[fill=black!5] (6,0) rectangle (16,7);
    \node at (11,6) {(I)};
    
    \begin{scope}
        \path[clip] (0,0) rectangle (22,7);
        \draw[thin, draw=black!40] (0,0.2) \u \u \u \u \u \u \u \u \u \u \u \d \d \d \d \d \d \d \d \d \d \d;
        \draw[thin, draw=black!40] (0,0)   \u \d \u \d \u \d \u \d \u \d \u \d \u \d \u \d \u \d \u \d \u \d;
    \end{scope}

  	%                       e         x                      f                
    \draw[blue, thick] (0,0.15)  % Q
    		\u \u \u \u \u \d \u \d \d coordinate (x) \d \u \u \d \d \d \d \u \u \d \u \d \d;
    \draw[red, thick] (0,0)  % P 
    		\u \u \d \u \d \u \u \d \u \u \d \d \u \u \u \d \d \u \d \d \d \d;
    \node[red] at (1,2) {$P'$};
    \node[blue] at (3,4) {$Q'$};
    \path (x) +(0.075,-0.075) node [label=above:{$x$}, fill, circle, inner sep=1pt] {};
  \end{tikzpicture}
  \caption{The injection for Case 1.}
\label{fig:case1}
\end{figure}

\item[Case 2.] Suppose the paths do not meet in region (I) and consider the paths $P, Q$ in the region (II) after position $f$. Imagine translating the fragment of path $Q$ after $f$ to $P$ so that their initial points in $f$ coincide. If the imaginary fragment intersects $P$, let $x$ be the first such point of intersection. Construct new paths $P',Q'$ by swapping the segments of $P$ and $Q$ between $f$ and $x$ (including $f$ itself).  See Figure \ref{fig:case2}.  If the new paths $P', Q'$ are both between $A$ and $B$ (so that $P'\in \B_e^f$ and $Q'\in B_f^e$), then we set $\varphi_{ef}(P, Q) = (P', Q')$.

%\TODO{Show some examples where Case 2 fails.}

\begin{figure}[ht]
  \centering % CASE 2
  \begin{tikzpicture}[scale=.25]
    \path[use as bounding box] (0,0) rectangle (22,9);
    \fill[fill=black!15] (7,0) rectangle ++(1,9) +(-.5,-1) node {$e$};
    \fill[fill=black!15] (12,0) rectangle ++(1,9) +(-.5,-1) node {$f$};
    \fill[fill=black!5] (13,0) rectangle (22,9);
    \node at (18,8) {(II)};
    
    \begin{scope}
        \path[clip] (0,0) rectangle (22,9);
        \draw[thin, draw=black!40] (0,0.2) \u \u \u \u \u \u \u \u \u \u \u \d \d \d \d \d \d \d \d \d \d \d;
        \draw[thin, draw=black!40] (0,0)   \u \d \u \d \u \d \u \d \u \d \u \d \u \d \u \d \u \d \u \d \u \d;
    \end{scope}

  	%                             e        f                            
    \draw[blue, thick] (0,0.15)  % Q
    		\u \u \u \u \d \u \u \d \u \d \u \u coordinate (fQ) \d \d \u \d \u \d \d \d \d \d;
    \draw[red, thick] (0,0)  % P 
    		\u \u \d \u \d \u \d \u \u \d \u \d coordinate (fP) \u \d \u \d \d \u \d \u \d \d;
    		
    \draw[red, thick, dotted] (fQ)+(0,.15)   \u \d \u \d \d coordinate (xQ) \u \d \u \d \d;
    \draw[blue, thick, dotted] (fP)          \d \d \u \d \u coordinate (xP) \d \d;
    \node[red] at (1,2) {$P$};
    \node[blue] at (3,4) {$Q$};
    \node[label=above:{$x$}, fill, circle, inner sep=1pt] at (xQ) {};
    \node[fill, circle, inner sep=1pt] at (xP) {};
  \end{tikzpicture}\qquad
  \begin{tikzpicture}[scale=.25]
    \path[use as bounding box] (0,0) rectangle (22,9);
    \fill[fill=black!15] (7,0) rectangle ++(1,9) +(-.5,-1) node {$e$};
    \fill[fill=black!15] (12,0) rectangle ++(1,9) +(-.5,-1) node {$f$};
    \fill[fill=black!5] (13,0) rectangle (22,9);
    \node at (18,8) {(II)};
    
    \begin{scope}
        \path[clip] (0,0) rectangle (22,9);
        \draw[thin, draw=black!40] (0,0.2) \u \u \u \u \u \u \u \u \u \u \u \d \d \d \d \d \d \d \d \d \d \d;
        \draw[thin, draw=black!40] (0,0)   \u \d \u \d \u \d \u \d \u \d \u \d \u \d \u \d \u \d \u \d \u \d;
    \end{scope}

  	%                             e        f                            
    \draw[blue, thick] (0,0.15)  % Q
    		\u \u \u \u \d \u \u \d \u \d \u \u coordinate (fQ) \u \d \u \d \d \d \d \d \d \d;
    \draw[red, thick] (0,0)  % P 
    		\u \u \d \u \d \u \d \u \u \d \u \d coordinate (fP) \d \d \u \d \u \u \d \u \d \d;
    		
    \draw[blue, thick, dotted] (fQ)             \d \d \u \d \u coordinate (xQ);
    \draw[red, thick, dotted] (fP)              \u \d \u \d \d coordinate (xP);
    \node[red] at (1,2) {$P'$};
    \node[blue] at (3,4) {$Q'$};
    \node[label=above:{$x$}, fill, circle, inner sep=1pt] at (xQ) {};
    \node[fill, circle, inner sep=1pt] at (xP) {};
  \end{tikzpicture}
\caption{The injection for Case 2.}
\label{fig:case2}
\end{figure}
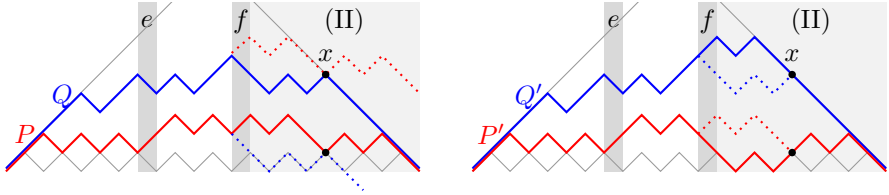

\item[Case 3.] Otherwise, perform the mirror image of the operation from Case 2 in the part (III) up to position $e$, as shown in Figure \ref{fig:case3} and again set $\varphi_{ef}(P, Q) = (Q', P')$. 

\begin{figure}[ht]
  \centering % CASE 3
  \begin{tikzpicture}[scale=.25]
    \path[use as bounding box] (0,0) rectangle (22,9);
    \fill[fill=black!15] (7,0) rectangle ++(1,9) +(-.5,-1) node {$e$};
    \fill[fill=black!15] (12,0) rectangle ++(1,9) + (-.5,-1) node {$f$};
    \fill[fill=black!5] (0,0) rectangle ++(7,9);
    \node at (3.5,8) {(III)};
    
    \begin{scope}
        \path[clip] (0,0) rectangle (22,9);
        \draw[thin, draw=black!40] (0,0.2) \u \u \u \u \u \u \u \u \u \u \u \d \d \d \d \d \d \d \d \d \d \d;
        \draw[thin, draw=black!40] (0,0)   \u \d \u \d \u \d \u \d \u \d \u \d \u \d \u \d \u \d \u \d \u \d;
    \end{scope}

  	%                             e        f                            
    \draw[blue, thick] (0,0.15)  % Q
    		\u \u \u \u \d \u \u \d coordinate (eQ) \u \d \u \u coordinate (fQ) \d \d \d \u \u \d \d \d \d \d;
    \draw[red, thick] (0,0)  % P 
    		\u \u \d \u \d \d \u \u coordinate (eP) \u \d \u \d coordinate (fP) \u \d \u \d \d \u \d \u \d \d;
    		
    \draw[red, thick, dotted] (fQ)           \u \d \u \d \d \u;
    \draw[blue, thick, dotted] (fP)          \d \d \d;
    
    \draw[red, thick, dotted] (eQ) \dl\dl\ul coordinate (xQ);
    \draw[blue, thick, dotted] (eP)\ul\dl\dl coordinate (xP);
    \node[red] at (1,2) {$P$};
    \node[blue] at (3,4) {$Q$};
    \node[label=above:{$x$}, fill, circle, inner sep=1pt] at (xQ) {};
    \node[fill, circle, inner sep=1pt] at (xP) {};
  \end{tikzpicture}\qquad
  \begin{tikzpicture}[scale=.25]
    \path[use as bounding box] (0,0) rectangle (22,9);
    \fill[fill=black!15] (7,0) rectangle ++(1,9) +(-.5,-1) node {$e$};
    \fill[fill=black!15] (12,0) rectangle ++(1,9) + (-.5,-1) node {$f$};
    \fill[fill=black!5] (0,0) rectangle ++(7,9);
    \node at (3.5,8) {(III)};
    
    \begin{scope}
        \path[clip] (0,0) rectangle (22,9);
        \draw[thin, draw=black!40] (0,0.2) \u \u \u \u \u \u \u \u \u \u \u \d \d \d \d \d \d \d \d \d \d \d;
        \draw[thin, draw=black!40] (0,0)   \u \d \u \d \u \d \u \d \u \d \u \d \u \d \u \d \u \d \u \d \u \d;
    \end{scope}
    
  	%                             e        f                            
    \draw[blue, thick] (0,0.15)  % Q
    		\u \u \u \u \d \d \u \u coordinate (eQ) \u \d \u \u coordinate (fQ) \d \d \d \u \u \d \d \d \d \d;
    \draw[red, thick] (0,0)  % P 
    		\u \u \d \u \d \u \u \d coordinate (eP) \u \d \u \d coordinate (fP) \u \d \u \d \d \u \d \u \d \d;
    		
%    \draw[red, thick, dotted] (fQ)+(0,.15)   \u \d \u \d \d \u;
%    \draw[blue, thick, dotted] (fP)          \d \d \d;
    
    \draw[blue, thick, dotted] (eQ) \ul\dl\dl coordinate (xQ);
    \draw[red, thick, dotted] (eP)  \dl\dl\ul coordinate (xP);
    \node[red] at (1,2) {$Q'$};
    \node[blue] at (3,4) {$P'$};
    \node[label=above:{$x$}, fill, circle, inner sep=1pt] at (xQ) {};
    \node[fill, circle, inner sep=1pt] at (xP) {};
  \end{tikzpicture}
\caption{The injection for Case 3.  Case 2 fails because the translated segment of $Q$ leaves the valid region before intersecting $Q$.}
\label{fig:case3}
\end{figure}
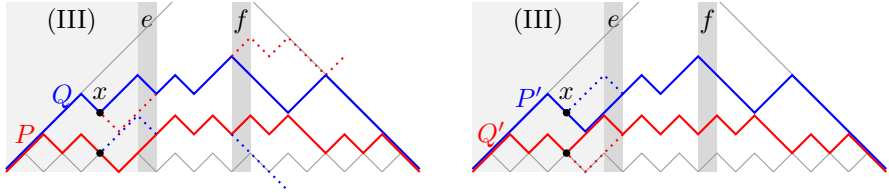
\end{description}

%\TODO{Should the analysis section include more details?  e.g., proof that the cases always apply when we say they do, or more formal proof that the function is injective? Or are these obvious enough?
%
%-- I think it's clear enough from what we say in the next two paragraphs.}
Case 1 covers all pairs of paths $P,Q$ such that $P$ is above $Q$ at $e$ and below it at $f$ or vice versa.  Case 2 covers all remaining pairs of paths with $P$ above $Q$ at $f$, since in this case the fragment of $Q$ must intersect $P$ and both resulting paths $(P',Q')$ are between $P$ and $Q$.  Finally, Case 3 covers all still remaining pairs of paths with $P$ below $Q$ at $e$.  Hence every pair of paths $P,Q$ is covered by one of these cases.

%For every position $t$, if both $P,Q$ were up-steps (resp. down-steps) at $t$, then they have the same property at $t$ in $P', Q'$. Thus the function $\varphi_{ef}$ is well-defined for any minor $\B$ of Catalan matroid. 

Finally we must argue that $\varphi_{ef}$ is injective.  Given $P',Q'$ we must be able to determine which case above was applied to produce them, and from this it is simple to recover $P$ and $Q$.  Since we know $e,f$ it is easy to identify the regions (I), (II), (III) in the three cases.  If $P',Q'$ intersect in region (I) then they must have come from Case 1, as neither of the later cases can produce such an intersection.  If they do not, first try to apply the inverse operation for Case 2, which is the same as the forwards operation; if this was not possible for the starting paths $P,Q$ then it is not possible for $P',Q'$ either (as performing the transformation of Case 3 cannot cause the transformation of Case 2 to become valid if it was not already), so if we have a pair of paths $P',Q'$ for which the second move is possible and does result in lattice paths between $A$ and $B$ then we know we must have arrived at it through Case 2.  Otherwise, we were in Case 3.  See Figure \ref{fig:recovery}.

\begin{figure}[ht]
  \centering
  \begin{tikzpicture}[scale=.25]
    \path[use as bounding box] (0,0) rectangle (22,8);
    \fill[fill=black!15] (7,0) rectangle ++(1,8);% +(-.5,-1) node {$e$};
    \fill[fill=black!15] (12,0) rectangle ++(1,8);% +(-.5,-1) node {$f$};
    
    \begin{scope}
        \path[clip] (0,0) rectangle (22,8);
        \draw[thin, draw=black!40] (0,0.2) \u \u \u \u \u \u \u \u \u \u \u \d \d \d \d \d \d \d \d \d \d \d;
        \draw[thin, draw=black!40] (0,0)   \u \d \u \d \u \d \u \d \u \d \u \d \u \d \u \d \u \d \u \d \u \d;
    \end{scope}

  	%                             e        f                            
    \draw[blue, thick] (0,0.15)  % Q
    		\u \u \u \u \d \d \u \u coordinate (eQ) \u \d \u \u coordinate (fQ) \d \d \d \u \u \d \d \d \d \d;
    \draw[red, thick] (0,0)  % P 
    		\u \u \d \u \d \u \u \d coordinate (eP) \u \d \u \d coordinate (fP) \u \d \u \d \d \u \d \u \d \d;
    		
    \draw[red, thick, dotted] (fQ)           \u \d \u \d \d \u;
    \draw[blue, thick, dotted] (fP)          \d \d \d;
    
    \draw[blue, thick, dotted] (eQ) \ul\dl\dl coordinate (xQ);
    \draw[red, thick, dotted] (eP)  \dl\dl\ul coordinate (xP);
    \node[red] at (1,2) {$Q'$};
    \node[blue] at (3,4) {$P'$};
%    \node[label=above:{$x$}, fill, circle, inner sep=1pt] at (xQ) {};
%    \node[fill, circle, inner sep=1pt] at (xP) {};
  \end{tikzpicture}
\caption{An example of recovering the case from the result.  Since the paths do not intersect in region (I) we must have come from Case 2 or 3.  Case 2 fails because the translated segment of $P'$ leaves the valid region before rejoining $Q'$, so the paths must have come from Case 3.  Note that the transformation $\phi_{ef}$ is usually not onto.}
\label{fig:recovery}
\end{figure}
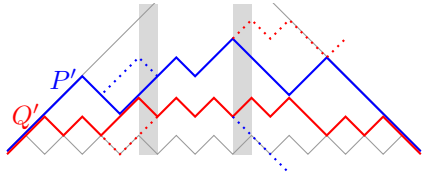
\end{proof}

\begin{remark}
    Since Case 1 can be applied whenever $P$ is above $Q$ at $e$ and below at $f$ or vice versa, Case 2 can be applied whenever $P$ is above $Q$ at both $e$ and $f$, and Case 3 can be applied whenever $P$ is below $Q$ at both $e$ and $f$, it is tempting to try to divide the cases more naturally according to the relative heights of $P$ and $Q$ at $e$ and $f$.  However, note that in Cases 2 and 3 (and sometimes Case 1) the resulting paths $P',Q'$ always have $P'$ above $Q'$ at both $e$ and $f$, and this more natural division of cases does not allow us to recover uniquely which transformation was applied. %See Figure \ref{fig:case-counterexample}.
%    \TODO{Figure for this?}
\end{remark}

\begin{remark}
    While the result of \cite{Bonin2006} that lattice path matroids are minor-closed shows that it suffices for us to prove negative correlation, it is worth noting that the same injection shows negative correlation directly for any minor of a lattice path matroid.  In particular, if we assume that $P$ and $Q$ above are both in the minor $\B_I^J$ (i.e., have up-steps at the indices of $I$ and down-steps at the indices in $J$) then the resulting paths $P'$ and $Q'$ again have this property. Moreover, the injection also implies that lattice path matroids are Rayleigh matroids (a class of matroids studied in \cite{Choe2006}).
\end{remark}

%%%%%%%%%%%%%%%%%%%%%%%%%%%%%%%%%%%%%%%%
\section{Lower bound} \label{sec:lower-bounds}
%%%%%%%%%%%%%%%%%%%%%%%%%%%%%%%%%%%%%%%%

To derive a lower bound on the mixing time of the Dyck random transposition chain we will need to refer to the definitions of the spectral gap and log-Sobolev constant.  Recall that for a Markov chain with state space $\Omega$, transition matrix $P$, and stationary distribution $\pi$
\begin{align*}
    \lambda \vcentcolon= \inf_f \frac{\mathcal{E}(f,f)}{\Var_\pi(f)}
    \qquad
    \text{and}
    \qquad
    \alpha \vcentcolon= \inf_f \frac{\mathcal{E}(f,f)}{\mathcal{L}_\pi(f)},
\end{align*}
where
\begin{align*}
    \mathcal{E}(f,f) &= \frac{1}{2} \sum_{x,y} (f(x) - f(y))^2 P(x,y) \pi(x)
    \qquad
    \text{and}
    \qquad\\
    \mathcal{L}_\pi(f) &= \sum_x f(x)^2(\log f(x)^2 - \log \E_\pi(f^2)) \pi(x)
\end{align*}
and the infima are taken over non-constant functions $f:\Omega\to \R$.

Using standard results we can (asymptotically) lower bound the mixing time by the relaxation time $\frac{1}{\lambda}$.
\marginnote{Can we give a better bound on $\alpha$ directly by using this test function?  It looks from Mathematica's plots as though $\mathcal{L}_\pi(f) \sim 4/n$, which probably isn't worth showing.}

\marginnote{Can we prove an $n^3\log n$ lower bound for Dyck adjacent transpositions, perhaps using a method similar to Wilson's lower bound for unconstrained adjacent transpositions?}

\begin{theorem}
    The spectral gap $\lambda$ and log-Sobolev constant $\alpha$ for the Dyck random transposition chain satisfy
    \[2\alpha \leq \lambda \leq \frac{4}{n}.\]
\end{theorem}

\begin{proof}
It is always the case that $\alpha \leq \lambda/2$ (see, e.g., \cite{Jerrum-Son}), so it suffices to show that $1/\lambda\geq n/4$.  

For $x\in \Omega$ a Dyck path of length $2n$, consider the function $f(x)$ giving the number of down-steps of $x$ at even indices.  The number of Dyck paths $x$ of length $2n$ with $f(x) = k$ is precisely the Narayana number $N(n,k) = \frac{1}{n}\binom{n}{k}\binom{n}{k-1}$ (see \cite{Sulanke2002}).   
% Proof: Convert a Dyck path to a non-crossing chord diagram by taking up-steps to be left endpoints and down-steps to be right endpoints.  From this chord diagram construct a non-crossing partition by adding each segment $(2i-1,2i)$ and considering the bounded regions thus created. Each region corresponds to a part consisting of the $i$s for segments on its boundary.  We have now given a bijection between Dyck paths and non-crossing partitions, and it is clear that a down-step at an even index corresponds to the right-most corner of a part and vice versa.  That is, the number of parts in the partition is the number of down-steps at even positions in the Dyck path.  It is well-known that the number of non-crossing partitions of $[n]$ into $k$ parts is $N(n,k)$.
Note that $f$ is $1$-Lipschitz, in that if $x,y$ differ by a single transposition (i.e.\ $P(x,y) > 0$) then $\abs{f(x) - f(y)} \leq 1$.  We can therefore bound $\lambda$ by noting that for a $1$-Lipschitz function we have $\mathcal{E}(f,f) \leq 1/2$, and so it suffices to give a lower bound on $\Var_\pi(f)$, % = \E_\pi(f^2) - \E_\pi(f)^2$. % and $\mathcal{L}(f) = \E_\pi(f^2 \ln f^2) - \E_\pi(f^2)\ln(\E_\pi(f^2))$.   
which is precisely the variance of the Narayana distribution with p.m.f.\ $p(k) = N(n,k)/C_n$.  This distribution is hypergeometric, and its variance is $\frac{(n+1)(n-1)}{4(2n-1)} \geq n/8$ (see, e.g., \cite{Johnson2005}).
\end{proof}

% Indeed, we can calculate
% \begin{align*}
%     \E_\pi(f) &= \frac{1}{C_n} \sum_{k=1}^n k N(n,k) = \frac{n+1}{2}\\
%     \E_\pi(f^2) &= \frac{1}{C_n} \sum_k k^2 N(n,k) = \frac{n^3 + 2n^2 - 1}{4n - 2}\\
%     \E_\pi(f^2 \ln f^2) &= \frac{1}{C_n} \sum_k k^2 \ln k^2 N(n,k) = ??.
% \end{align*}
%\marginnote{\cite{Johson2005} p.\ 291 immediately gives the variance as $\frac{(n+1)(n-1)}{4(2n-1)} \geq n/8$ (for $n\geq 2$).  Look at references there to find more.  See also p.\ 259.} Thus we find that the spectral gap is at most $(4 + o(1))/n$ and the log-Sobolev constant is at most...

\begin{corollary}
    The (total variation) mixing time of the Dyck random transposition walk is at least $\Omega(n)$.
\end{corollary}

%%%%%%%%%%%%%%%%%%%%%%%%%%%%%%%%%%%%%%%%
\section{Open questions} \label{sec:open-questions}
%%%%%%%%%%%%%%%%%%%%%%%%%%%%%%%%%%%%%%%%

Of course, the most immediate open question is whether the $O(n^2\log n)$ mixing time bound for Dyck transpositions is tight. 
\marginnote{Is there an easy $\Omega(n\log n)$ lower bound, or one we can cite?  We have an $O(1/n)$ bound on spectral gap.}
The general result for lattice path matroids is clearly not tight in every case (as it gives only order $n^2\log n$ mixing time for Bernoulli-Laplace, where we know that the true bound is order $n\log n$), but there are lattice path matroids for which the $O(n^2\log n)$ bound is tight\footnote{One example is the lattice path matroid $\L[A,B]$ where $A = (-+)^n$ and $B=(+-)^n$, which is equivalent to the random walk on the $n$-cube slowed down by a factor of $n$.}.  We believe that the Catalan matroid falls closer to the Bernoulli-Laplace side of the spectrum. 
%\marginnote{PT: $\Omega(n)$ is easy as a lower bound on the $L_2$-mixing time, using a test function $f$ in the definition of log-Sob or spectral gap. (Recall $\log_2 C_n \sim 2n$.) To get $\Omega(n\log n)$, we would have to do some work.}  
Either way, we hope in the future to work towards a more complete characterization of the mixing rates for various lattice path transposition chains.
\marginnote{dy: Can we prove a kind of monotonic result: say, if the lattice region A is contained in the region B, then mixing time for A is at least as for B?}
\marginnote{EC: The chain on the region between $+-(+-)^{n-1}$ and $-+(+-)^{n-1}$ mixes in order $n^2$ time, which is faster than the $n^2 \log n$ mixing time for the region between $(+-)^n$ and $(-+)^n$. Might still work for spectral gap?}

It is also not clear what consequences the new bound for Dyck paths might have for mixing on other Catalan structures.  Many well-known Catalan structures suggest natural Markov chains, but it is striking how different they can be---even though there are natural bijections between classes of Catalan structures, they frequently do not preserve natural notions of distance, making standard comparison arguments tricky.  There seems to be no obvious way to leverage even a hypothetical $O(n\log n)$ mixing time result for Dyck transpositions to yield improved bounds for any other Catalan chain. It is worth noting that the classical ($2n+1$-to-one many) Chung-Feller cyclic lemma suggests projecting the chain consisting of all transpositions on the (Bernoulli-Laplace) space of $\{\pm 1\}^{2n+1}$, with $(n+1)$ 1's and $n$ -1's, onto the set of Dyck paths. This projection chain is in fact a Markov chain, as was observed by \cite{CRT10}, and as such inherits a lower bound on the spectral gap of order $1/n$, from that of the Bernoulli-Laplace chain; however, many of the moves in the resulting projection chain are not particularly natural when viewed from the point of Dyck paths, and once again comparisons do not seem to help very much.

%The result of \cite{McShine1999} bounding the mixing time for triangulation swaps (or equivalently node rotations in binary trees) relies on a comparison to the adjacent Dyck transposition chain.  Given that the random Dyck transposition chain mixes faster, it is tempting to compare the triangulation chain to that instead, but even assuming the conjectured $O(n\log n)$ mixing time for random Dyck transpositions we have been unable to improve on the bound of \cite{McShine1999}.

An alternate approach would be to come up with a general scheme to bound mixing times which could be modified for any Catalan structure.  For example, it is well-known that the Catalan numbers satisfy the recurrence
\[C_n = \sum_{i=1}^{n} C_{i-1} C_{n-i},\]
and this is evidenced by a recursive structure in almost every combinatorial realization of the Catalan sequence.  It is common to exploit recursive structure in bounding mixing times, and indeed both the triangulation bound of \cite{MRS98} and the balanced matroid bounds of \cite{Feder1992, Jerrum-Son, Jerrum2004} exploit different recursive structures of those objects.  Unlike the recurrence above, the recursive structures used in those papers seem very particular to the Catalan realizations at hand, and do not generalize well to other Catalan chains.% and we are still looking for some way to exploit the more universal Catalan recursion above.

Finally, we have mentioned above the conjecture that the basis exchange walk is rapidly mixing for \emph{all} matroids.  As evidence, it would be interesting to show rapid mixing for some class of matroids which is not a subclass of balanced matroids.  One promising candidate might be the class of transversal matroids, for which the basis exchange walk seems closely related to the well-studied problem of walks on bipartite graph matchings.

%\TODO{Is our result at all related to the result of Dyer-M\"uller that the swap chain on perfect matchings of monotone bipartite graphs is fast? A monotone bipartite graph has vertex set $[n]\cup[n']$ with neighborhoods $N(i) = [a_i, b_i]'$ for some $a_1\leq \dots \leq a_n$ and $b_1 \leq \dots \leq b_n$.  That is, these are precisely the incidence graphs for set systems corresponding to lattice path matroids, where bases of the matroid correspond to (equivalence classes of) maximum matchings.  Of course, the perfect matching chain is only interesting when a perfect matching exists, in which case the basis-exchange walk is trivial.}

%%%%%%%%%%%%%%%%%%%%%%%%%%%%%%%%%%%%%%%%
\section*{Acknowledgements}
%%%%%%%%%%%%%%%%%%%%%%%%%%%%%%%%%%%%%%%%
The authors thank the Institute for Mathematics and its Applications (IMA) for its excellent research atmosphere, generous support, and hospitality, which were all crucial to the collaboration resulting in the present work. The first two authors also acknowledge additional NSF support by way of the grant DMS-1407657.

\bibliographystyle{abbrv}
\bibliography{catalan}

\end{document}